\documentclass{amsart}
\usepackage[english]{babel} 
\usepackage[utf8]{inputenc}
\usepackage[T1]{fontenc}

\usepackage{amsmath,amsfonts,amssymb,amsthm}
\usepackage{bbm,mathrsfs,MnSymbol}
\usepackage[longnamesfirst,numbers]{natbib}
\usepackage[margin]{fixme}
\usepackage{setspace}
\usepackage{enumitem}         
\usepackage{lscape}
\usepackage{stmaryrd} 
\usepackage{todonotes}
\usepackage{color}
\definecolor{darkgreen}{rgb}{0,0.55,0}
\usepackage{hyperref}
\hypersetup{pdfstartview={Fit},    
  pdftitle={European Americans},    
  pdfauthor={Martin Larsson, Marvin S. Mueller, Josef Teichmann},     
  pdfsubject={},   
  pagebackref=true,
  colorlinks=true,
  linkcolor=blue,
  citecolor=darkgreen}

\newtheorem{thm}{Theorem}[section]

\theoremstyle{definition}

\theoremstyle{remark}

\newtheorem{rmk}[thm]{Remark}

\numberwithin{equation}{section}

\newtheorem{theorem}[thm]{Theorem}
\newtheorem{proposition}[thm]{Proposition}
\newtheorem{corollary}[thm]{Corollary}
\newtheorem{lemma}[thm]{Lemma}

\theoremstyle{definition}
\newtheorem{definition}[thm]{Definition}

\theoremstyle{remark}
\newtheorem{remark}[thm]{Remark}

\newcommand\1{\mathbf{1}}

\newcommand\R{\mathbb{R}}

\newcommand\N{\mathbb{N}}

\DeclareMathOperator*{\argmin}{arg\,min} 
\DeclareMathOperator*{\argmax}{arg\,max} 
\DeclareMathOperator{\graph}{graph}

\providecommand{\abs}[1]{\left\lvert#1\right\rvert}
\providecommand{\norm}[2]{\left\lVert#1\right\rVert_{#2}}

\newcommand{\borel}[1]{\mathfrak B\left(#1\right)}

\newcommand{\QQ}{\mathbb{Q}}
\newcommand{\PP}{\mathbb{P}}

\newcommand{\EE}{\mathbb{E}}
\newcommand{\FF}{\mathbb{F}}

\newcommand{\cE}{\mathcal{E}}

\newcommand{\cH}{\mathcal{H}}

\newcommand{\cO}{\mathcal{O}}
\newcommand{\cP}{\mathcal{P}}

\newcommand{\cS}{\mathcal{S}}
\newcommand{\cL}{\mathcal{L}}

\newcommand{\cT}{\mathcal{T}}

\newcommand{\fP}{\mathfrak{P}}

\newcommand\ddt{\tfrac{\partial}{\partial t}}

\newcommand\ddx{\tfrac{\partial}{\partial x}}
\newcommand\ddxx{\tfrac{\partial^2}{\partial x^2}}

\newcommand\F{\mathcal{F}}

\renewcommand\d{\,\operatorname{d}\hspace{-0.05cm}}

\newcommand\dom{\mathcal{D}}                                      

\newcommand{\bsigma}{{\overline \sigma}}

\newcommand{\usigma}{\underline{\sigma}}

\title[On nonlinear optimal stopping]{Stopper-Controller Games embedded in Single-Player Control Problems}
\author{Martin Larsson, Marvin S. M\"uller, Josef Teichmann}
\address{Department of Mathematical Sciences, Carnegie Mellon University, Pittsburgh 15213, USA, Department of Mathematics, ETH Z\"urich, Switzerland}
\email{martinl@andrew.cmu.edu, marvin.s.mueller@gmail.com, jteichma@math.ethz.ch}

\thanks{The authors acknowledge support by the Swiss National Science Foundation through grant SNF $205121\_163425$.}


\date{\today}

\begin{document}
\begin{abstract}
  In 2002, Benjamin Jourdain and Claude Martini discovered that for a
  class of payoff functions, the pricing problem for American options
  can be reduced to pricing of European options for an appropriately
  associated payoff, all within a Black-Scholes framework. This discovery
  has been investigated in great detail by S\"oren Christensen, Jan
  Kallsen and Matthias Lenga in a recent work in 2020. In the present work we
  prove that this phenomenon can be observed in a wider context, and
  even holds true in a setup of non-linear stochastic processes. We
  analyse this problem from both probabilistic and analytic viewpoints.
  In the classical situation, Jourdain and Martini used this
  method to approximate prices of American put options. The broader
  applicability now potentially covers non-linear frameworks such as
  model uncertainty and controller-and-stopper-games.
\end{abstract}

\maketitle
\allowdisplaybreaks

\tableofcontents
\section{Introduction}
\label{sec:intro}

In the seminal work ``Approximation of American put prices by European
prices via an embedding method'', see \cite{jourdain2001american} and \cite{jourdain2002approximation}, Benjamin Jourdain and Claude Martini
discovered the following phenomenon: given a European option with
payoff $g$, then one can under a specific regularity condition associate an American option with
$ f \leq g $ such that the price functionals of the American and the
European option problem coincide. More precisely, let $ X $ denote a strong Markov process with continuous trajectories on some state space $E$ together with time horizon $ T > 0$,
and assume that there is a continuous function $t_*: E \to [0,T] $ such that
\[
f(x) := \inf_{0 \leq t \leq T} E_{t,x}[g(X_T)] = E_{t_*(x), x}[g(X_T)]
\]
holds true for $ x \in E $. Given $ x \in E $ and some time $ 0 \leq t \leq T $, and assume $ t_*(x) \geq t $. Let $ \tau $ be the first time when
\[
t_*(X_{\tau}) = \tau \, ,
\]
which exists due to continuity, then $ \tau $ is the optimal stopping time of the optimal stopping problem associated to $f$. Indeed for every stopping time $ t \leq \eta \leq T $
\begin{align*}
  E_{t,x}[g(X_T)] & = E_{t,x}[E_{\eta,X_{\eta}}[g(X_{T})]] \geq \sup_{t \leq \eta \leq T} E_{t,x}[f(X_\eta)] \\
  & \geq E_{t,x}[f(X_{\tau})] = E_{t,x}[E_{\tau,X_{\tau}}[g(X_{T})]]=E_{t,x}[g(X_T)]  
\end{align*}
holds true, whence equality. A similarly easy, but considerably more far-reaching proof, will be
presented in Section~\ref{sec:prob}. This important observation has
not been taken up by the literature until S\"oren Christensen, Jan
Kallsen and Matthias Lenga finally (almost) proved in ``Are American
options European after all'', see \cite{lenkalsor:20}, that actually the American put in the
Black-Scholes framework has a value function, which can be considered
coming from a European pricing problem, whence the discovery of
Jourdain-Martini arrived in the very heart of American option
theory. It is the purpose of this article to extend this idea to a
non-linear and multi-variate framework: we show that a class of stopper-controller games
can be embedded into a single-player optimal control problem, which
presents a substantial generalization of the results of Jourdain-Martini
or Christensen-Kallsen-Lenga.

We use two perspectives here: first, we consider a probabilistic
approach where the cost and reward structure is driven by an
underlying continuous-time stochastic process. The value of the game
is calculated from the cost-and-reward structure by the use of a
system of evaluation mappings, which is related to non-linear
expectations. The key assumption will be the dynamic programming
principle. Second, from a more analytic viewpoint, we show the
respective statements in terms of fully non-linear partial
differential equations. More precisely, we show that the free boundary
problem occurring in non-linear optimal stopping problems can be
solved in terms of a fully non-linear partial differential equation
(PDE), e.\,g. of classical Hamilton-Jacobi-Bellman (HJB)-type.

The presented framework naturally covers the optimal stopping problems
such as American option pricing, where the market model is driven by a
strong Markov process. Beyond that, the possible non-linearity of the
evaluation maps allows for parameter uncertainty and to cover
stopper-controller games.

In Section~\ref{sec:prob} we will present the main results in an
abstract probabilistic setup. In Section~\ref{sec:PDE} we then discuss
the PDE viewpoint. We then show how to apply the abstract results to
important examples: the benchmark case is American option pricing in
the Black-Scholes market and will be discussed in
Section~\ref{sec:BS}. In Section~\ref{sec:nonlinLevy} we then discuss
non-linear Levy processes which arise for instance in
stopper-controller games and volatility uncertainty. We finally
discuss the example of finite state spaces in
Section~\ref{sec:finite}, in which the HJB equation for the value
functions becomes an ordinary differential equation.

Relevant notation is collected in Appendix~\ref{sec:notation}.

\section{On optimal stopping in a nonlinear framework}
\label{sec:prob}
\subsection{Setup}

We consider a path space $(\Omega,\F)$ with values in some Polish
space $E$, and the coordinate process $X$ together with a filtration
$(\F_t)_{t\geq 0}$ generated by $X$. Let
$B\colon [0,\infty) \to (0,\infty)$ be a function such that $B$ and
$1/B$ are bounded on bounded sets.

We fix a terminal time $T\in (0,\infty)$ and for $t\in [0,T]$ denote
by $\cT_{t,T}$ a class of $[t,T]$-valued random variables
being the set of admissible stopping rules. For consistency we assume that if $s,t\in [0,T]$ with $s\leq t$ then
$\cT_{t,T}\subseteq \cT_{s,T}$.

Moreover, let $\cH$ be a linear subspace of $\cL^0(\Omega)$ such that
$\1_\Omega\in \cH$ and
\begin{equation}
  \label{eq:20}
  \forall Y\in \cH, Z\in \cL^\infty(\Omega):\qquad YZ\in \cH.
\end{equation}
In particular, $\cL^\infty(\Omega)\subseteq\cH$. Elements of the space
$\cH$ have the role of integrable random variables and elements of the
space $\cT_{t,T}$ play the role of stopping times but we emphasize
that no probability measures have been introduced until now. The
evaluation of these random variables will instead be performed by
possibly non-linear expectations.

Let $\{\cE_{t,x}\}_{(t,x)\in [0,T]\times E}$ be a family of (possibly
nonlinear) pre-expectations on $\cH$, this is, a family of functions
\begin{equation}
  \label{eq:9}
  \cE_{t,x}\colon \cH \to \R
\end{equation}
for $(t,x) \in [0,T]\times E$ such that
\begin{enumerate}[label=($\cE$.\Roman*)]
\item\label{i:const} for each constant $c\in \R$, it holds that
  \[ \cE_{t,x}[c] = c.\]
\item\label{i:monoton} for each $(t,x)\in [0,T]\times E$ and each $Y$,
  $Y'\in \cH$ with $Y\leq Y'$ it holds that
  \[ \cE_{t,x}[Y] \leq \cE_{t,x}[Y'].\]
\end{enumerate}

Note that for each $(t,x)\in [0,T]\times E$, $\cE_{t,x}$ is a
non-linear expectation in the sense of
\cite[p. 168]{peng2004nonlinear} and sometimes also called nonlinear
pre-expectation. It is worth stressing that we do not impose any
sublinearity or homogeneity a priori but \ref{i:const} and
\ref{i:monoton} are sufficient to obtain, e.g., certain continuity properties.
For instance taking $Y$, $Y'\in \cH$ we obtain that
\begin{equation*}
   \cE_{t,x}[\pm(Y-Y')]  \leq \cE_{t,x}[ \norm{Y-Y'}{\cL^\infty}]= \norm{Y-Y'}{\cL^\infty}
\end{equation*}
holds true, if $ \norm{Y-Y'}{\cL^\infty} $ is finite.

By assumptions on $\cH$ it holds for each $A\in \F$ that
$\1_{A}\in \cH$. We will call an event $A\in \F$ to be a \emph{full
  $(t,x)$-scenario}, if
\begin{equation}
  \label{eq:22}
  \forall Y\in \cH:\quad \cE_{t,x}[Y\1_{A}] = \cE_{t,x}[Y].
\end{equation}

In addition to~\ref{i:monoton} we impose the following assumption,
which together then ensure a strict monotonicity:
\begin{enumerate}[label=($\cE$.\Roman*), resume]
\item\label{i:monoton_strict} if for $(t,x)\in [0,T]\times E$ and $Y$,
  $Y'\in \cH$ with $Y\leq Y'$ it holds that
  \[ \cE_{t,x}[Y] = \cE_{t,x}[Y'],\] then the event $\{Y=Y'\}$ is a
  full $(t,x)$-scenario.
\end{enumerate}

To stress the role of $\cE_{t,x}$ as $(t,x)$-evaluation we
additionally assume:
\begin{enumerate}[label=($\cE$.\Roman*), resume]
\item\label{i:txeval} for each $(t,x)\in [0,T]\times E$, $\{X_t=x\}$
  is a full $(t,x)$-scenario.
\end{enumerate}
This assumption is consistent with the setup for linear Markov
processes.

\subsection{Value functions and optimal stopping}
For two measurable functions
$\varrho\colon [0,T]\times [0,T] \times \Omega \times C([0,T];\R) \to
\R$ and $g\colon E\to \R$ we define the value function
\begin{equation}
  \label{eq:value_EU}
  v(t,x) := B_t \cE_{t,x}\left[ \varrho(t,T, X, B) + B_T^{-1}g(X_T)\right].
\end{equation}

\begin{remark}
  In typical situations it holds that for $s\leq t$,
  \[ \varrho(t,s, .,.) \equiv 0,\] and $\varrho(s,t,X,B)$ represents
  the (discounted) pre-maturity cashflow. The payoff structure at
  maturity time $T$ is given by the function $g$. Classical examples
  in optimal control are
  \[ \varrho(s,t, \omega, b) = \int_s^t B_r^{-1} c(r,\omega_r) \d r\]
  for a sufficiently nice function $c\colon [0,T]\times E\to \R$. In
  option pricing, $g$ is the terminal payoff and $\varrho$ describes
  possible rebate payments or costs for holding the option.
\end{remark}

To ensure time-consistency we impose the dynamic programming
principle:
\begin{enumerate}[label=($\cE$.\Roman*), resume]
\item\label{i:meas} for each $\tau \in\cT_{0,T}$ the value function
  defined in \ref{eq:value_EU} satisfies
  \[v(\tau, X_\tau), \inf_{0\leq t\leq T} v(t, X_\tau), \sup_{0\leq
      t\leq T} v(t,X_\tau) \in
    \cH\] 
\item\label{i:dpp} for each $(t,x)\in [0,T]\times E$,
  $\tau \in \cT_{t,T}$, it holds that
  \[v(t,x) = B_t\cE_{t,x}\left[\varrho(t, \tau, X, B) + B_\tau^{-1}
      v(\tau, X_\tau)\right].\]
\end{enumerate}

The assumptions on the time-infimum and -supremum allow us to define
the payoff functions
\begin{equation}
  \label{eq:1_sec2}
  f(x):= \inf_{0\leq t\leq T} v(t,x),\qquad h(x):=
  \sup_{0\leq t\leq T} v(t,x),
\end{equation}
for $x\in E$.

The main object of this section are now the non-linear optimal
stopping problems with value at state $(t,x)$ given by:
\begin{equation}
  \label{eq:valuefcts_stopping}
  \begin{aligned}
    u(t,x) &:= \sup_{\tau \in \cT_{t,T}} B_t
    \cE_{t,x}\left[ \varrho(t,\tau, X, B) +  B_\tau^{-1} f(X_\tau) \right],\\
    w(t,x) &:= \inf_{\tau \in \cT_{t,T}} B_t \cE_{t,x}\left[
      \varrho(t,\tau, X, B) + B_\tau^{-1} h(X_\tau) \right],
  \end{aligned}
\end{equation}

\begin{remark}
  If we assume that, in addition, for each $(t,x)\in [0,T]\times E$,
  $\cE_{t,x}$ is sub-linear and positively homogeneous, then by
  \cite{nutz2013constructing} there exist sets of probability measures
  $\cP(t,x)$, such that
  \[ \forall Y\in \cH:\quad \cE_{t,x}[Y] = \sup_{\PP\in \cP(t,x)}
    \EE^{\PP}[Y],\] or,
  \[ \forall Y\in \cH:\quad \cE_{t,x}[Y] = \inf_{\PP\in \cP(t,x)}
    \EE^{\PP}[Y].\] In this situation, $u$ and $w$ are the value
  functions of controller-stopper games, where the first player
  selects a control measure $\PP_{t,x} \in \cP(t,x)$ (and thus the
  dynamics of the reward process $X$), and the second player stops the
  game at a time $\tau \in\cT_{t,T}$. The in-game cash flow is
  $\varrho$ and the payoff at $\tau$ is respectively $f(X_\tau)$ and
  $h(X_\tau)$. Classical examples for such games are American options
  from holder and issuer perspective. We consider more details for
  such examples and choices of $\cP(t,x)$ in
  Sections~\ref{sec:finite},~\ref{sec:BS} and~\ref{sec:nonlinLevy}.
\end{remark}

We get the following general relation between the value functions:
\begin{proposition}\label{prop:ubounds}
  For all $(t,x)\in [0,T]\times E$,
  \[ u(t,x) \leq v(t,x)\leq w(t,x).\]
\end{proposition}
\begin{proof}
  Let $(t,x)\in [0,T]\times E$. By definition of $f$ and $v$ we have
  for each $\tau \in \cT_{t,T}$
  \begin{equation}
    \label{eq:domination}
    f(X_\tau) 
    \leq v(\tau, X_\tau)
  \end{equation}
  By monotonicity~\ref{i:monoton} and the DPP~\ref{i:dpp} this yields
  \begin{equation}
    \begin{aligned}
      B_t \cE_{t,x}[\varrho(t,\tau, X, B) + B_\tau^{-1} f(X_\tau)]
      &\leq  B_t \cE_{t,x}[\varrho(t,\tau, X, B) +B_\tau^{-1}  v(\tau, X_\tau)] \\
      &= v(t,x).
    \end{aligned}
  \end{equation}  
  Taking the supremum over $\tau\in \cT_{t,T}$ on the left hand side
  we finally get
  \[u(t,x)\leq v(t,x).\] The proof for the other estimate is similar:
  By monotonicity~\ref{i:monoton} and the DPP~\ref{i:dpp} this yields
  \begin{equation}
    \begin{aligned}
      B_t \cE_{t,x}\left[\varrho(t,\tau, X, B) +
        B_\tau^{-1}h(X_\tau)\right]
      &\geq  B_t \cE_{t,x}[\varrho(t,\tau, X, B) + B_\tau^{-1} v(\tau, X_\tau)] \\
      &= v(t,x).
    \end{aligned}\qedhere
  \end{equation}    
\end{proof}
\begin{remark}\label{rmk:onesideddpp}
  It is clear from the proof that instead of assuming~\ref{i:dpp} it
  would be sufficient to assume the respective one-sided version. That
  is, to obtain
  \[u \leq v\] it suffices to assume that for each
  $(t,x)\in [0,T]\times E$, $\tau \in \cT_{t,T}$, it holds that
  \[v(t,x) \geq B_t\cE_{t,x}\left[\varrho(t, \tau, X, B) +
      B_\tau^{-1}v(\tau, X_\tau)\right].\] For
  \[ v \leq w\] it is sufficient that for each
  $(t,x)\in [0,T]\times E$, $\tau \in \cT_{t,T}$, it holds that
  \[v(t,x) \leq B_t\cE_{t,x}\left[\varrho(t, \tau, X, B) +
      B_\tau^{-1}v(\tau, X_\tau) \right].\]
\end{remark}

To get a deeper understanding of the relation between the value
functions $u$, $v$ and $w$, we we replace \ref{i:meas} by the stronger
assumption that $v$ is continuous. In this case, we define for each
$x\in E$ the sets
\begin{align*}
  t_*(x) &:= \argmin\{v(t,x) \colon 0\leq t\leq T\},\\
  t^*(x) &:= \argmax\{v(t,x) \colon 0\leq t\leq T\}.
\end{align*}
Recall from optimal stopping that the continuation regions are defined
as
\begin{align*}
  C_* &:= \{(t,x)\in [0,T]\times E \colon u(t,x) >
        f(x)\},\\
  C^* &:= \{(t,x)\in [0,T]\times E \colon w(t,x) < h(x)\}.
\end{align*}
To formulate a bound on these regions, we will use the notation that
for $\bar t \in \{t_*, t^*\}$:
\begin{equation}
  \label{eq:15}
  \graph(\bar t) := \{ (t,x)\in [0,T]\times E \colon t\in \bar t(x)\}.
\end{equation}
\begin{corollary}\label{cor:contreg}
  Assume that $v$ is continuous, then for all $x\in E$,
  \[ t_*(x) \subseteq \{ t\in [0,T] \colon u(t,x) = f(x)\},\] and
  \[t^*(x) \subseteq \{ t\in [0,T] \colon w(t,x) = h(x)\} .\] In
  particular,
  \[ C_*\subseteq \graph(t_*)^c ,\quad\text{and}\quad C^*\subseteq
    \graph(t^*)^c.\]
\end{corollary}
\begin{proof}
  Applying Proposition~\ref{prop:ubounds} yields for $x\in E$ and
  $t\in t_*(x)$,
  \[ f(x) = v(t,x) \geq u(t,x) \geq f(x).\] For $t\in t^*(x)$,
  \[h(x) = v(t,x) \leq w(t,x) \leq h(x).\qedhere\]
\end{proof}

We now characterize optimal solutions of the stopping problems, making
the relations between $u$, $v$ and $w$ more precise.
\begin{theorem}\label{thm:wvu}
  Assume that $v$ is continuous, and let $(t,x)\in [0,T]\times E$.
  \begin{enumerate}[label=(\roman*)]
  \item\label{i:vu} For $\tau \in \cT_{t,T}$, it holds that
    \begin{equation*}
      v(t,x) = \cE_{t,x}[\varrho(t,\tau,X,B) + B_\tau^{-1} f(X_T)],
    \end{equation*}
    if, and only if, $\{\tau \in t_*(X_\tau)\}$ is a full
    $(t,x)$-scenario. In this case, then $u(t,x) = v(t,x)$.
  \item\label{i:wv} For $\tau \in \cT_{t,T}$, it holds that
    \begin{equation*}
      v(t,x) = \cE_{t,x}[\varrho(t,\tau,X,B) + B_\tau^{-1} h(X_T)],
    \end{equation*}
    if, and only if, $\{\tau \in t^*(X_\tau)\}$ is a full
    $(t,x)$-scenario. In this case, then $w(t,x) = v(t,x)$.
  \end{enumerate}
\end{theorem}
\begin{remark}
  Note that for any $\tau \in \cT_{0,T}$:
  \[\{\tau \in t_*(X_\tau)\} = \{(\tau, X_\tau)\in \graph(t_*)\}\] and
  \[\{\tau \in t^*(X_\tau)\} = \{(\tau,X_\tau)\in \graph(t^*)\}\]
  which are measurable since $\graph(t_*)$ and $\graph(t^*)$ are
  closed by Lemma~\ref{lem:graphclosed}.
\end{remark}
\begin{remark}
  There is a quantitative and more general version of Theorem
  \ref{thm:wvu} in case that all appearing quantities are
  non-negative, i.e.~$ \rho \geq 0 $ and $g \geq 0$. We shall
  formulate it in the case of~\ref{i:vu}. Instead of assuming that
  $\{\tau \in t_*(X_\tau)\}$ is a full $(t,x)$-scenario, we assume
  that there is $ \epsilon \geq 0 $ such that
  \[
    B_t \cE_{t,x}[ \1_{\tau\in t_*(X_\tau)} \left(\varrho(t,\tau, X,
      B) + B_\tau^{-1}v(\tau, X_\tau)\right)] \geq v(t,x) - \epsilon
    \, .
  \]
  Conceptually speaking this is the case when
  $\{\tau \in t_*(X_\tau)\}$ is a very likely scenario (in this robust
  setting). The assumption of Theorem~\ref{thm:wvu} is included with
  $ \epsilon = 0 $.

  Then we can conclude, following the lines of the proof below, that
  \[
    v(t,x) \geq u(t,x) \geq v(t,x) - \epsilon \, ,
  \]
  i.e.~the value function $u$ is sandwiched between $v$ and
  $v-\epsilon$. The consequences of this phenomenon will be discussed
  in upcoming work.

  In a completely analogous manner this is also true for~\ref{i:wv}.
\end{remark}
\begin{proof}
  We will only show the proof of~\ref{i:vu}, since~\ref{i:wv} is
  similar.

  Assume that $\{\tau\in t_*(X_\tau)\}$ is a full $(t,x)$-scenario. We
  directly compute, using~\ref{i:dpp}:
  \begin{equation}
    \begin{aligned}
      B_t& \cE_{t,x}[\varrho(t,\tau, X, B) + B_\tau^{-1}f(X_\tau)] \\
      &=B_t \cE_{t,x}[ \1_{\tau\in t_*(X_\tau)}\left(\varrho(t,\tau,
        X, B) +
        B_\tau^{-1} f(X_\tau)\right)] \\
      &=B_t \cE_{t,x}[ \1_{\tau\in t_*(X_\tau)} \left(\varrho(t,\tau, X, B) +  B_\tau^{-1}v(\tau, X_\tau)\right)] \\
      &=B_t \cE_{t,x}\left[\varrho(t,\tau, X, B) + B_\tau^{-1} v(\tau, X_\tau)\right] \\
      &= v(t,x).
    \end{aligned}
  \end{equation}
  Taking the supremum over all stopping times and using
  Proposition~\ref{prop:ubounds} we get
  \[ u(t,x) \geq v(t,x) \geq u(t,x).\]

  Now, assume that
  \[ B_t \cE_{t,x} [\varrho(t,\tau, X, B) + B_\tau^{-1}f(X_\tau)] =
    v(t,x).\] Then, by~\ref{i:dpp}
  \begin{equation}
    \cE_{t,x} [\varrho(t,\tau, X, B) + B_\tau^{-1}f(X_\tau)] = B_t^{-1} v(t,x) =  \cE_{t,x} [\varrho(t,\tau, X, B) + B_\tau^{-1}v(\tau, X_\tau)].
  \end{equation}
  On the other hand, we know that $f(y)\leq v(s,y)$ for all
  $(s,y)\in [0,T]\times E$ by definition, and hence by~\ref{i:monoton}
  \begin{equation}
    \label{eq:16}
    \varrho(t,\tau, X, B) + B_\tau^{-1}f(X_\tau) \leq  \varrho(t,\tau, X, B) + B_\tau^{-1}v(\tau, X_\tau).
  \end{equation}
  The strict monotonicity~\ref{i:monoton_strict} then yields that
  $\{f(X_\tau) = v(\tau, X_\tau)\}$ is a full $(t,x)$-scenario. To
  finish the proof we just have to recall that
  \[ \{f(X_\tau) = v(\tau,X_\tau)\} = \{ \tau \in
    t_*(X_\tau)\}. \qedhere\]
\end{proof}

In order to ensure existence of such an optimal stopping time, we need
a deeper knowledge of the sets $t_*(x)$ and $t^*(x)$, for $x\in E$. We
will delay the study of some of the properties to
Section~\ref{ssec:tstar} below, and discuss the implications first.

\begin{corollary}\label{cor:tstaroptimal_convex}
  Assume that $X$ and $v$ are continuous and that $t_*(x)$ is convex
  for all $x\in E$. Let $(t,x)\in [0,T]\times E$ be such that
  $t\leq t_*(x)$ and define
  $\tau_*:= \inf\{s>t\,\colon\, s\in t_*(X_s)\}$. Then,
  $\tau_*\in [t,T]$ and if $\tau_*\in \cT_{t,T}$, then
  \[ u(t,x) = B_t\cE_{t,x}[\varrho(t,\tau_*, X,B) + B_{\tau_*}^{-1}
    f(X_{\tau_*})] = v(t,x).\] Similarly, define
  $\tau^*:= \inf\{s>t\,\colon\, s\in t^*(X_s)\}$. If instead of
  $t\leq t_*(x)$ we have $t\leq t^*(x)$, then $\tau^*\in [t,T]$ and if
  $\tau^*\in \cT_{t,T}$, then
  \[ w(t,x) = B_t\cE_{t,x}[\varrho(t,\tau^*, X, B) +
    B_{\tau^*}^{-1}f(X_{\tau^*})] = v(t,x).\]
\end{corollary}
\begin{proof}
  First, if $t< \min t_*(X_t)$, then $[t,T]\ni s\mapsto (s,X_s)$
  intersects with $\graph(t_*)$ by Lemma~\ref{lem:graph_intersect}
  which will happen the first time at $\tau_*$, so that
  $\tau_* \in t_*(X_{\tau_*})$. Hence, by~\ref{i:txeval} for all
  $Y\in \cH$,
  \begin{multline}
    \label{eq:18}
    \cE_{t,x}[Y] = \cE_{t,x}[Y\1_{X_t=x} \left(\1_{t<\min
        t_*(X_{\tau_*})} + \1_{t\in t_*(X_t)}\right)] \\= \cE_{t,x}[Y
    \left(\1_{t<\min t_*(X_{\tau_*}), \tau_*\in X_{\tau_*}}+ \1_{t\in
        t_*(X_t), \tau_*\in t_*(X_{\tau_*})}\right)]\\ =
    \cE_{t,x}[Y\1_{\tau_*\in t_*(X_{\tau_*})}].
  \end{multline}
  Hence, Theorem~\ref{thm:wvu}.\ref{i:vu} yields the result for
  $\tau_*$. The proof for $\tau^*$ works in the same way using
  Theorem~\ref{thm:wvu}.\ref{i:wv}.
\end{proof}

\begin{corollary}\label{cor:ctspath}
  Assume that $X$ and $v$ are continuous and that there exists a
  continuous function $\theta\colon E\to [0,T]$, such that
  \begin{equation}
    \label{eq:17}
    \theta(x) \in t_*(x) \, ,
  \end{equation}
  for all $ x \in E $. Let $(t,x) \in [0,T]\times E$ be such that $t\leq \theta(x)$ and set
  $\tau := \inf\{s\geq t\colon s= \theta(X_s)\}$. Then,
  $\{\tau\in t_*(X_\tau)\}$ is a full $(t,x)$-scenario and if in
  addition $\tau_*\in \cT_{0,T}$ then
  \begin{equation}
    \label{eq:19}
    u(t,x) = B(t)\cE_{t,x}[\varrho(t,\tau,X,B) +
    B_\tau^{-1}f(X_\tau)] = v(t,x). 
  \end{equation}
  Similarly, the result holds true when replacing $(t_*, u, f)$ by
  $(t^*, w, h)$.
\end{corollary}
\begin{proof}
  The function $\eta\colon s\mapsto \theta(X_{s})\vee t$ is continuous
  and mapping $[t,T]$ into $[t,T]$, if $t\leq \theta(X_t)$. By Brouwer
  fixed point theorem $\eta$ admits a non-empty set of fixed points
  $A$, which is closed as a level set of a continuous function. By
  definition and $t\leq\theta(X_t)$ it holds that $\tau = \min A$ and
  thus $\tau = \theta(X_\tau) \in t_*(X_\tau)$. This means that for
  each $Y\in \cH$
  \begin{multline}
    \label{eq:21}
    \cE_{t,x}[Y] = \cE_{t,x}[ Y\1_{t\leq \theta(x)}] = \cE_{t,x}[ Y\1_{t\leq \theta(X_t)}] \\
    = \cE_{t,x}[ Y\1_{t\leq \theta(X_t)}\1_{\tau \in t_*(X_\tau)}] =
    \cE_{t,x}[ Y\1_{\tau \in t_*(X_\tau)}].
  \end{multline}
  Here, we used \ref{i:txeval} in the second and in the last step. But
  this means that the assumptions of Theorem~\ref{thm:wvu}.\ref{i:vu}
  are satisfied for all $(0,x)$, $x\in E$ which yields the result.
\end{proof}

\subsection{Properties of $t_*$ and $t^*$}
\label{ssec:tstar}
\begin{lemma}
  \label{lem:graphclosed}
  Assume that $v$ is continuous. Then, $\graph(t_*)$ and $\graph(t^*)$
  are closed, $x\mapsto \max t_*(x)$ and $x\mapsto \max t^*(x)$ are
  upper semicontinuous and $x\mapsto \min t_*(x)$ and
  $x\mapsto \min t^*(x)$ are lower semicontinuous.
\end{lemma}
\begin{proof}
  Replacing $v$ by $-u$ the statements for $t_*$ and $t^*$ are
  equivalent, so that, without loss of generality, we focus on $t_*$.
  Let $(t_n,x_n)_{n\in \N}$ be a sequence in $\graph(t_*)$ such that
  $\lim_{n\to \infty} (t_n,x_n) = (t,x)$ for some
  $(t,x)\in [0,T]\times E$. By definition of $t_*$, it for all
  $n\in \N$ and all $s\in [0,T]$ that
  \[v(t_n,x_n) \leq v(s,x_n). \] Hence, by continuity of $v$, for all
  $s\in[0,T]$
  \[ v(t,x) = \lim_{n\to\infty} v(t_n,x_n) \leq \lim_{n\to\infty}
    v(s,x_n) = v(s,x).\] This yields $(t,x)\in \graph(t_*)$ which is
  thus closed.

  As level set of a continuous function, the set $t_*(x)$ is closed
  for each $x\in E$, in fact even compact. In particular, the maximum
  and minimum of $\graph(t_*)$ are well defined.  Let
  $(x_n)_{n\in \N}$ be a sequence in $E$ such that for some
  $(\bar t,\bar x) \in [0,T]\times E$ it holds that
  $\lim_{n\to\infty} x_n = \bar x$ and for $t_n:= \max t_*(x_n)$, it
  holds that $\lim_{n\to\infty} t_n = \bar t$. By continuity of $v$ we
  get $(t_n, x_n)\in\graph(t_*)$ and by closedness of the graph of
  $t_*$ also $(\bar t,\bar x)\in \graph(t_*)$. Thus,
  $\bar t\leq \max t_*(\bar x)$, which yields upper semicontinuity. In
  the same way we get lower semicontinuity of $x\to \min t_*(x)$
\end{proof}

\begin{lemma}
  \label{lem:graph_intersect}
  Assume that $v$ is continuous and that $t_*(x)$ is convex for all
  $x\in E$. Let $(t,x) \in E$ with $t<\min t_*(x)$ and $y\in E$. Then,
  any continuous path in $[0,T]\times E$ from $(t,x)$ to $(T,y)$
  intersects $\graph(t_*)$. The same holds true for $t^*$.
\end{lemma}
\begin{proof}
  Let
  $[0,1]\ni \lambda \mapsto (t(\lambda), x(\lambda)) \in [0,T]\times
  E$ be a continuous path from $(t,x)$ to $(T,y)$ and
  \[ \bar \lambda = \inf\{\lambda \in [0,1] \colon t(\lambda) \geq
    \min t_*(x(\lambda))\}, \quad \bar t:= t(\bar \lambda),\quad \bar
    x := x(\bar \lambda).\] Then, $\bar \lambda >0$ since
  $\graph(t_*)$ is closed by Lemma~\ref{lem:graphclosed}, and
  $\bar\lambda \le 1$ since $t(1) = T\geq \min t_*(y)$ for all
  $y\in E$. Moreover, for all $\lambda <\bar \lambda$,
  \[ t(\lambda) < \min t_*(x(\lambda)) \leq \max t_*(x(\lambda)).\]
  Thus,
  \[ \bar t = \lim_{\lambda \nearrow \bar\lambda} t(\lambda)\leq
    \limsup_{\lambda \nearrow \bar \lambda} \max t_*(x(\lambda)) \leq
    \limsup_{x \to \bar x} \max t_*(x)) \leq \max t_*(\bar x).\] Here,
  we used the upper semicontinuity of $x\mapsto \max t_*(x)$ which is
  shown in Lemma~\ref{lem:graphclosed}. Let $(\lambda_n)_{n\in \N}$ be
  a non-increasing sequence in $[0,1]$ such that
  $\lambda_n\geq \bar \lambda$ and
  $t(\lambda_n)\geq \min t_*(x(\lambda_n))$ for all $n\in \N$. Then,
  \begin{equation}
    \label{eq:times}
    \bar t = \lim_{n\to\infty} t(\lambda_n) \geq \liminf_{n\to\infty}
    \min t_*(x(\lambda_n)) \geq \min t_*(x(\bar \lambda)), 
  \end{equation}
  by lower semicontinuity of $x\mapsto \min t_*(x)$; see
  Lemma~\ref{lem:graphclosed}. Summarizing, we have seen that
  \[ \min t_*(\bar x) \leq \bar t\leq \max t_*(\bar x),\] and
  convexity of $t_*(\bar x)$ yields $\bar t\in t_*(\bar x)$.
\end{proof}

\begin{remark}
  If $t\mapsto v(t,x)$ has a unique minimizer, then $t_*(x)$ is a
  singleton and hence convex. Note that if $t_*(x)$ is a singleton for
  all $x$, then $x\to t_*(x)$ is continuous since the graph of $t_*$
  is closed; as seen in Lemma~\ref{lem:graphclosed}. Here, we abuse
  the notation and write $t_*(x)$ for the element it contains.
\end{remark}
\begin{remark}
  If $t\mapsto v(t,x)$ is convex, then $t_*(x)$ is convex. If
  $t\mapsto v(t,x)$ is concave, then $t_*(x)$ is either a singleton
  and thus convex, or $t_*(x)$ has cardinality $2$ and is not
  convex. Vice versa, $t^*(x)$ is a singleton or has cardinality $2$
  if $t\mapsto v(t,x)$ is convex and $t^*(x)$ is convex, if
  $t\mapsto v(t,x)$ is concave.
\end{remark}

\section{Kolmogorov Equations and Free Boundary Problems}
\label{sec:PDE}

Often, the value function of a European option is the (unique)
solution of the Kolmogorov backward PDE. The value function of an
American option is given by a free boundary problem, where the optimal
stopping boundary separates the two regimes (continuation / stopping)
of the equation. In this section we study this representation of the value function.

\subsection{Solving a free boundary problem}
We consider a state space $E\subseteq \R^d$ for some $d\geq 1$ and a
possibly non-linear operator $L\colon \dom(L)\to C([0,T]\times E;\R)$
such that $\dom(L)\subseteq C^{1,2}([0,T],E;\R)$, and
$g\colon E\to\R$. We moreover assume that $L$ is local, that means
that for each $x\in E$ and each neighborhood $U$ of $x$, and each $v$,
$\tilde u\in \dom(L)$ such that $u= \tilde u$ on $U$ it holds that
$Lv(x) = L\tilde v(x)$. Let $T\in (0,\infty)$.

To be in line with Appendix~\ref{sec:visco}, we assume that
$\dom(L) = C^{1,2}([0,T],E;\R)$ and that there exists a proper
function
$F\colon [0,T]\times E\times \R \times \R^d \times S_d \to \R$, such
that
\[ L(v)(t,x) = -F(t,x,v(t,x), D_x v(t,x), D_{x}^2 v(t,x)).\] for
$u\in \dom(L)$ and $(t,x) \in [0,T]\times E$. From technical point of
view this is, however, not important in the following.

We now consider the following backward equation,
\begin{equation}
  \label{eq:kolmogorov}
  \begin{alignedat}{2}
    -\ddt v(t,x) - L(v)(t,x) &= 0,&\quad (t,x)&\in [0,T]\times E,\\
    v(T,x) &= g(x),
  \end{alignedat}
\end{equation}
and the free boundary problem

\begin{equation}
  \label{eq:fbp_equations}
  \begin{alignedat}{2}
    -\ddt u(t,x) - L(u)(t,x) &\geq 0,&\quad (t,x) &\in [0,T]\times E,\\
    -\ddt u(t,x) - L(u)(t,x) &= 0,&\quad (t,x) &\in C,\\
    u(t,x) &= f(x),& (t,x) &\in D,\\
    C= \{(t,x) \in [0,T)\times E &\colon u(t,x) > f(x)\},&D:=&C^c\\
    \text{(continuous fit)}\qquad  u(t,x) &=  f(x),& (t,x) &\in \partial C,\\
    \text{or (smooth fit)}\qquad \nabla_x u(t,x) &= \nabla_x f(x),&
    (t,x) &\in \partial C.
  \end{alignedat}
\end{equation}

We emphasize that a solution of~\eqref{eq:fbp_equations}, is a continuous
function $u\colon [0,T]\times E$, which satisfies $u(t,x) = f(x)$ on
$D$, and, in the sense of Definition~\ref{def:visc}, $u$ is a
viscosity supersolution on $[0,T]\times E$ and a viscosity solution on
$C$ of the equation
\[ -\ddt u(t,x) - L(u)(t,x) = 0.\] Here, $C$ and $D$ are defined above
and depend on $u$ themselves.

\begin{theorem}
  \label{thm:PDEequality}
  Assume that~\eqref{eq:kolmogorov} admits a continuous viscosity
  solution $v\colon [0,T] \times E \to \R$. Define
  \begin{equation}
    \label{eq:1_sec3}
    f(x) := \inf_{0\leq t\leq T} v(t,x),\qquad \vartheta(x) := \inf\{\argmin\{v(t,x)\colon t\in [0,T]\}\}\wedge T,
  \end{equation}
  for $x\in E$, as well as
  \begin{equation*}
    C_t := \left\{ x\in E \colon t<  \vartheta(x)\right\},\qquad C:= \bigcup_{t\in [0,T]} \left( \{t\} \times C_t\right),
  \end{equation*}
  and for $(t,x)\in [0,T]\times E$,
  \begin{equation}
    \label{eq:2_sec2}
    u(t,x) :=
    \begin{cases}
      v(t,x),& (t,x) \in C,\\
      f(x),& (t,x) \notin C,
    \end{cases}
  \end{equation}
  and assume that $\vartheta$ is continuous. Then, $u$ solves
  \eqref{eq:fbp_equations} in the viscosity sense and satisfies the continuous
  fit condition. Moreover, if $\theta$ and $v$ are of class $C^1$,
  then the smooth fit condition holds true.
\end{theorem}
\begin{proof}
  First, note that $C$ is open as the preimage of $(0,\infty)$ under
  the continuous map $(t,x)\mapsto \theta(x)-t$.
  
  Let $(t,x)\in C$ and $U$ be a neighborhood of $(t,x)$ such that
  $U\subset C$. Let $\phi \in C^2([0,T]\times E; \R)$ be such that
  $\phi(t,x) = u(t,x)$ and $\phi \geq u$ in a neighborhood $U$ of
  $(t,x)$. By definition of $C$ this means $\phi(t,x) = v(t,x)$ and
  $\phi\geq v$ in the neighborhood $U\cap C$. Since $v$ is a
  subsolution and $L$ is local,
  \[ -\ddt \phi - L \phi \leq 0,\qquad \text{on }C,\] and thus, $u$ is
  a viscosity subsolution on $C$. In the same way we can see that $u$
  is also a viscosity supersolution on $C$ and thus, $u$ is a
  viscosity solution in $C$.
  
  Let $(t,x)\in C$, then $t< \vartheta(x)$ and thus, $t<T$ and
  $u(t,x) = v(t,x) > f(x)$. On the other hand, let
  $(t,x) \in [0,T] \times E$ be such that $u(t,x)> f(x)$. Then,
  $(t,x) \notin C^c$ by definition of $u$ which proves
  \[C = \{ (t,x) \in [0,T) \times E \colon u(t,x) > f(x) \}.\] It
  remains to verify the continuous and smooth fit conditions.
  
  Note that
  $\overline{C} \subseteq \{ (t,x) \in [0,T]\times E \colon t\leq
  \vartheta(x)\}$. In fact, if $(t,x) \in \overline{C}$, then there
  exists $(t_n,x_n)\in C$, $n\in \N$ such that
  $\lim_{n\to\infty} (t_n,x_n) = (t,x)$, and thus, using continuity of
  $\vartheta$,
  \begin{equation}
    \label{eq:3_vartheta}
    t-\vartheta(x) = \lim_{n\to\infty} t_n -\vartheta(x_n) \leq 0,    
  \end{equation}
  since $t_n -\vartheta(x_n) < 0$ for all $n\in \N$.

  Let now $(t,x) \in \partial C$. Since $C$ is open we get that
  $(t,x) \notin C$ and thus $t = \vartheta(x)$. Moreover, when
  $(t_n,x_n) \in C$ such that $\lim_{n\to\infty} (t_n,x_n) = (t,x)$,
  then by continuity of $v$,
  \begin{equation}
    \label{eq:6_limit}
    \lim_{n\to\infty} u(t_n,x_n) = \lim_{n\to\infty}v(t_n,x_n) = v(t,x) = v(\vartheta(x), x) = f(x),
  \end{equation}
  which is the continuous fit condition. If $\theta$ is
  differentiable, then
  \begin{equation}
    \label{eq:7_vartheta}
    \nabla_x f(x) = \nabla_x \vartheta(x) (\partial_t v)(\vartheta(x),x) + (\nabla_x v)(\vartheta(x), x) 
    = (\nabla_x v)(\vartheta(x),x),
  \end{equation}
  since $(\partial_t v)(\vartheta(x),x) = 0$, if
  $(t,x) \in \partial C$. On the other hand,
  \begin{equation}
    \label{eq:8_limit}
    \lim_{n \to\infty} \nabla_x u(t_n,x_n) =  \lim_{n \to\infty} \nabla_x v(t_n,x_n) = \nabla_x v(t,x)
  \end{equation}
  by continuity of $\nabla_x v$. It remains
  to verify the supersolution property on
  $D = [0,T]\times E \setminus C$. Let $(t,x) \in D$ and take
  $\phi \in C^{1,2}$ such that $\phi(t,x) = u(t,x)$ and
  $\phi\leq u(t,x)$ on a neighborhood $U$ of $(t,x)$. We first
  consider the case when $t= \theta(x)$. Then, $\phi\leq v$ on $U$. In
  fact, for $(s,y)\in U\cap C$ it holds that
  $\phi(s,y) \leq u(s,y) = v(s,y)$ and for $(s,y)\in U\cap D$ we get
  $\phi(s,y) \leq u(s,y) = f(y) \leq v(s,y)$. Moreover, by the
  continuous fit we have that
  \[\phi(t,x) = u(\theta(x),x) = f(x) = v(\theta(x),x) = v(t,x).\]
  Since $v$ is in particular a supersolution of~\eqref{eq:kolmogorov},
  we get that
  \[ -\ddt \phi(t,x) - L\phi(t,x) \geq 0.\]
  
  For the situation when $t>\theta(x)$ we know that $(t,x)\in D^\circ$
  and assume without loss of generality that $U\subseteq D$. Then,
  define
  \[\tilde U:= U - (t-\theta(x),0)\]
  which is a neighborhood of $(\theta(x),x)$ in $[0,T]\times E$. Let
  $\tilde\phi \in C^{1,2}$ be such that for all $(s,y)\in \tilde U$:
  \[ \tilde \phi(s,y) = \phi(s + t-\theta(x), y).\] Then,
  $\tilde \phi(\theta(x),x) = \phi(t,x) = u(t,x) = f(x) =
  v(\theta(x),x)$. Moreover, note that with $t_s := s+t-\theta(x)$ we
  have $(t_s,y) \in U\subseteq D$ and thus
  \[\tilde \phi(s,y) = \phi(t_s,y) \leq u(t_s,y) = f(y) \leq
    v(s,y), \] by definition of $f$. Again, the supersolution property
  of $v$ and the fact that $\ddt$ and $L$ are local and
  shift-invariant yields
  \[ -\ddt \phi(t,x) - L\phi(t,x) = \left(-\ddt \tilde \phi(s,y) - L
      \tilde \phi(s,y)\right) \vert_{s=\theta(x), y = x} \geq
    0.\qedhere\]
\end{proof}
\begin{remark}
  By straight forward modification of the second paragraph in the
  proof one can see that Theorem~\ref{thm:PDEequality} also holds true
  when replacing the solution concept of \emph{viscosity solutions} by
  \emph{classical} or \emph{strong solutions}.
\end{remark}

We also get the similiar result for the free boundary problem, with
terminal condition $h\colon E\to\R$:
\begin{equation}
  \label{eq:fbp}
  \begin{alignedat}{2}
    -\ddt w(t,x) - L(w)(t,x) &\leq 0,&\quad (t,x) &\in [0,T]\times E,\\
    -\ddt w(t,x) - L(w)(t,x) &= 0,&\quad (t,x) &\in C,\\
    w(t,x) &= h(x),& (t,x) &\in D,\\
    C= \{(t,x) \in [0,T)\times E &\colon w(t,x) < h(x)\},&D:=&C^c\\
    \text{(continuous fit)}\qquad  w(t,x) &=  h(x),& (t,x) &\in \partial C,\\
    \text{or (smooth fit)}\qquad \nabla_x w(t,x) &= \nabla_x h(x),&
    (t,x) &\in \partial C.
  \end{alignedat}
\end{equation}
\begin{theorem}
  \label{thm:PDEequalitySup}
  Assume that~\eqref{eq:kolmogorov} admits a continuous viscosity
  solution $v\colon [0,T] \times E \to \R$. Define
  \begin{equation}
    \label{eq:1_h}
    h(x) := \sup_{0\leq t\leq T} v(t,x),\qquad \vartheta(x) := \inf\{\argmax\{v(t,x)\colon t\in [0,T]\}\}\wedge T,
  \end{equation}
  for $x\in E$, as well as
  \begin{equation*}
    C_t := \left\{ x\in E \colon t<  \vartheta(x)\right\},\qquad C:= \bigcup_{t\in [0,T]} \left( \{t\} \times C_t\right),
  \end{equation*}
  and for $(t,x)\in [0,T]\times E$,
  \begin{equation}
    \label{eq:2_w}
    w(t,x) :=
    \begin{cases}
      v(t,x),& (t,x) \in C,\\
      h(x),& (t,x) \notin C,
    \end{cases}
  \end{equation}
  and assume that $\vartheta$ is continuous. Then, $u$ solves
  \eqref{eq:fbp} in the viscosity sense and satisfies the continuous
  fit condition. Moreover, if $\theta$ and $v$ are of class $C^1$,
  then the smooth fit condition holds true.
\end{theorem}
\begin{proof}
  First, note that $C$ is open as the preimage of $(0,\infty)$ under
  the continuous map $(t,x)\mapsto \theta(x)-t$.
  
  Let $(t,x)\in C$ and $U$ be a neighborhood of $(t,x)$ such that
  $U\subset C$. Let $\phi \in C^2([0,T]\times E; \R)$ be such that
  $\phi(t,x) = w(t,x)$ and $\phi \leq w$ in a neighborhood $U$ of
  $(t,x)$. By definition of $C$ this means $\phi(t,x) = v(t,x)$ and
  $\phi\leq v$ in the neighborhood $U\cap C$. Since $v$ is a
  supersolution and $L$ is local,
  \[ -\ddt \phi - L \phi \geq 0,\qquad \text{on }C,\] and thus, $u$ is
  a viscosity supersolution on $C$. In the same way we can see that
  $w$ is also a viscosity subsolution on $C$ and thus, $w$ is a
  viscosity solution in $C$.
  
  Let $(t,x)\in C$, then $t< \vartheta(x)$ and thus, $t<T$ and
  $w(t,x) = v(t,x) < h(x)$. On the other hand, let
  $(t,x) \in [0,T] \times E$ be such that $w(t,x)< h(x)$. Then,
  $(t,x) \notin C^c$ by definition of $w$ which proves
  \[C = \{ (t,x) \in [0,T) \times E \colon w(t,x) < h(x) \}.\] It
  remains to verify the continuous and smooth fit conditions.
  
  Note that
  $\overline{C} \subseteq \{ (t,x) \in [0,T]\times E \colon t\leq
  \vartheta(x)\}$. In fact, if $(t,x) \in \overline{C}$, then there
  exists $(t_n,x_n)\in C$, $n\in \N$ such that
  $\lim_{n\to\infty} (t_n,x_n) = (t,x)$, and thus, using continuity of
  $\vartheta$,
  \begin{equation}
    \label{eq:3}
    t-\vartheta(x) = \lim_{n\to\infty} t_n -\vartheta(x_n) \leq 0,    
  \end{equation}
  since $t_n -\vartheta(x_n) < 0$ for all $n\in \N$.

  Let now $(t,x) \in \partial C$. Since $C$ is open we get that
  $(t,x) \notin C$ and thus $t = \vartheta(x)$. Moreover, when
  $(t_n,x_n) \in C$ such that $\lim_{n\to\infty} (t_n,x_n) = (t,x)$,
  then by continuity of $v$,
  \begin{equation}
    \label{eq:6}
    \lim_{n\to\infty} w(t_n,x_n) = \lim_{n\to\infty}v(t_n,x_n) = v(t,x) = v(\vartheta(x), x) = h(x),
  \end{equation}
  which is the continuous fit condition. If $\theta$ is
  differentiable, then
  \begin{equation}
    \label{eq:7_nabla}
    \nabla_x h(x) = \nabla_x \vartheta(x) (\partial_t v)(\vartheta(x),x) + (\nabla_x v)(\vartheta(x), x) 
    = (\nabla_x v)(\vartheta(x),x),
  \end{equation}
  since $(\partial_t v)(\vartheta(x),x) = 0$, if
  $(t,x) \in \partial C$. On the other hand,
  \begin{equation}
    \label{eq:8}
    \lim_{n \to\infty} \nabla_x w(t_n,x_n) =  \lim_{n \to\infty} \nabla_x v(t_n,x_n) = \nabla_x v(t,x)
  \end{equation}
  by continuity of $\nabla_x v$. It remains
  to verify the supersolution property on
  $D = [0,T]\times E \setminus C$. Let $(t,x) \in D$ and take
  $\phi \in C^{1,2}$ such that $\phi(t,x) = w(t,x)$ and
  $\phi\geq w(t,x)$ on a neighborhood $U$ of $(t,x)$. We first
  consider the case when $t= \theta(x)$. Then, $\phi\geq v$ on $U$. In
  fact, for $(s,y)\in U\cap C$ it holds that
  $\phi(s,y) \geq w(s,y) = v(s,y)$ and for $(s,y)\in U\cap D$ we get
  $\phi(s,y) \geq w(s,y) = h(y) \geq v(s,y)$. Moreover, by the
  continuous fit we have that
  \[\phi(t,x) = w(\theta(x),x) = h(x) = v(\theta(x),x) = v(t,x).\]
  Since $v$ is in particular a subsolution of~\eqref{eq:kolmogorov},
  we get that
  \[ -\ddt \phi(t,x) - L\phi(t,x) \leq 0.\]
  
  For the situation when $t>\theta(x)$ we know that $(t,x)\in D^\circ$
  and assume without loss of generality that $U\subseteq D$. Then,
  define
  \[\tilde U:= U - (t-\theta(x),0)\]
  which is a neighborhood of $(\theta(x),x)$ in $[0,T]\times E$. Let
  $\tilde\phi \in C^{1,2}$ be such that for all $(s,y)\in \tilde U$:
  \[ \tilde \phi(s,y) = \phi(s + t-\theta(x), y).\] Then,
  $\tilde \phi(\theta(x),x) = \phi(t,x) = w(t,x) = h(x) =
  v(\theta(x),x)$. Moreover, note that with $t_s := s+t-\theta(x)$ we
  have $(t_s,y) \in U\subseteq D$ and thus
  \[\tilde \phi(s,y) = \phi(t_s,y) \geq w(t_s,y) = h(y) \geq
    v(s,y), \] by definition of $h$. Again, the subsolution property
  of $v$ and the fact that $\ddt$ and $L$ are local and
  shift-invariant yields
  \[ -\ddt \phi(t,x) - L\phi(t,x) = \left(-\ddt \tilde \phi(s,y) - L
      \tilde\phi(s,y)\right) \vert_{s=\theta(x), y = x} \leq
    0.\qedhere\]
\end{proof}

\subsection{A variational equality}

An alternative formulation of the free boundary problem is the
variational equality
\begin{equation}
  \label{eq:variational}
  \begin{alignedat}{2}
    \min\{-\ddt u(t,x) - L(u)(t,x), u(t,x) -g(x)\} &= 0,&\quad (t,x) &\in [0,T]\times E,\\
    u(T,x) &= f(x),& x &\in E.
  \end{alignedat}
\end{equation}
We get of course the same result in this formulation.
\begin{theorem}\label{thm:variational}
  Assume that $L$ is a local operator, that $g$ is such that
  \eqref{eq:kolmogorov} admits a viscosity solution
  $v\colon [0,T] \times E \to \R$ which is continuous. Define
  \begin{equation}
    \label{eq:1_f}
    f(x) := \inf_{0\leq t\leq T} v(t,x),\qquad \vartheta(x) := \inf\{\argmin\{v(t,x)\colon t\in [0,T]\}\}\wedge T,
  \end{equation}
  for $x\in E$, as well as
  \begin{equation*}
    C_t := \left\{ x\in E \colon t<  \vartheta(x)\right\},\qquad C:= \bigcup_{t\in [0,T]} \left( \{t\} \times C_t\right),
  \end{equation*}
  and for $(t,x)\in [0,T]\times E$,
  \begin{equation}
    \label{eq:2_u}
    u(t,x) :=
    \begin{cases}
      v(t,x),& (t,x) \in C,\\
      f(x),& (t,x) \notin C,
    \end{cases}
  \end{equation}
  and assume that $\vartheta$ is continuous. Then, $u$ is a continuous
  viscosity solution of~\eqref{eq:variational} in the viscosity
  sense.
  
\end{theorem}
\begin{proof}
  The proof that $C$ is open works as above.
  
  Fix $(t,x)\in [0,T]\times E$.  Let
  $\phi \in C^{1,2}([0,T]\times E; \R)$ be such that
  $\phi(t,x) = u(t,x)$ and $\phi \geq u$ in a neighborhood $U$ of
  $(t,x)$. If $(t,x)\in C$, we have $\phi(t,x) = v(t,x)$ and
  $\phi\geq v$ on the neighborhood $U\cap C$. Hence, since $v$ is a
  subsolution of~\eqref{eq:kolmogorov} and $L$ is local:
  \[ -(\ddt \phi + L \phi)(t,x) \leq 0.\] Hence,
  \[\min\{ - (\ddt \phi + L \phi)(t,x), \phi(t,x) - f(x)\} \leq 0.\]
  If $(t,x) \in C^c$, then $\phi(t,x) = u(t,x) = f(x)$ and
  \[\min\{ - (\ddt \phi + L \phi)(t,x), \phi(t,x) - f(x)\} \leq 0.\]

  We have shown that $u$ is a subsolution of~\eqref{eq:variational}.
  
  We have seen in the proof of Theorem~\ref{thm:PDEequality} that $u$
  is a supersolution of
  \[ -\ddt u(t,x) - L(u)(t,x) = 0.\] Finally, for
  $(t,x)\in [0,T]\times E$ and $\phi \in C^{1,2}$ such that
  $\phi(t,x) = u(t,x)$ and $\phi \leq u$ in a neighborhood of $(t,x)$,
  we get that $\phi(t,x) = u(t,x) \geq f(x)$, and that
  \[ \min\{ -\ddt \phi(t,x) - L(\phi)(t,x), \phi(t,x) - f(x) \} \geq
    0.\] Thus, $\phi$ is also a viscosity supersolution
  of~\eqref{eq:variational}. The continuity and fit conditions was
  shown in the proof of Theorem~\ref{thm:PDEequality}.
\end{proof}

\section{Black-Scholes Model}
\label{sec:BS}
The first example we want to discuss is option pricing in the
classical Black-Scholes model. $X$ will be the process of asset prices
- following a geometric Brownian motion - and $B$ a save bank
account. In fact, the $(t,x)$-evaluations will be linear expectations.

\subsection{Setup}
We set $E:= (0,\infty)$, $\Omega:= C([0,T];(0,\infty))$ and take $X$
as the coordinate process. We keep a fixed interest rate $r>0$ by
setting $B_t := e^{rt}$, $t\geq 0$. Also, we fix a volatility
parameter $\sigma >0$. Let $(\F_{t\geq 0})$ be the right continuous
filtration generated by $X$ and for $s,t\in [0,T]$ with $s\leq t$ let
$\cT_{s,t}$ be the set of all stopping times with values in $[s,t]$.

Let $(\PP_{t,x})_{(t,x)\in [0,T]\times E}$ be the strong Markov family
such that for all $(t,x)\in [0,T]\times E$ it holds that:
\begin{enumerate}
\item for all $s\in [0,t]$,
  \[ \PP_{t,x}[X_s = x] = 1,\]
\item $(B_tB_{t\wedge s}^{-1}X_s)_{s\in [0,T]}$ is a
  $\PP_{t,x}$-martingale,
\item $\langle X_s\rangle_s = \sigma(s-t)$ $\PP_{t,x}$-almost surely,
  for all $s\in (s,T]$.
\end{enumerate}
In other words: Under $\PP_{t,x}$, $X_s=x$ for $s\leq t$ and $X_s$ is
a geometric Brownian motion with drift $r$, for $s>t$.

We denote by $\EE_{t,x}$ the expectation under $\PP_{t,x}$. Clearly,
$(\EE_{t,x})_{(t,x)\in [0,T]\times E}$ fulfills \ref{i:const},
\ref{i:monoton}, \ref{i:monoton_strict} and \ref{i:txeval} from
Section~\ref{sec:prob}. Let us stress that an event $A\in \F$ is a
full $(t,x)$-scenario if and only if $A$ holds $\PP_{t,x}$-almost
surely.

\subsection{Option Pricing}

We consider the valuation problem of European and American options:

We focus on a bounded and continuous terminal payoff function
$g\colon \R_+\to \R$ at time $T$ for the European option and set
$\varrho :=0$ and denote its value function by $v$. As above, $u$ is
the value function of an American option with payoff
$f:= \inf_{0\leq t\leq T}v(t,.)$. This is,
\begin{equation}
  \label{eq:23}
  v(t,x) := \cE_{t,x}[e^{-r(T-t)} g(X_T)],\quad  u(t,x):=
  \sup_{\tau\in \cT_{t,T}}\cE_{t,x}[e^{-r(\tau-t)}f(X_\tau)]
\end{equation}

It is well known that $v$ is continuous and bounded, and thus we can
define
\[ t_*(x) := \argmin\{ u(t,x) \,\colon\, t\in [0,T]\},\] for all
$x\in \R_+$, and we also consider the continuation region
\[ C_*:= \{(t,x)\in [0,T]\times \R_+,\colon\, u(t,x) > f(x)\}.\]

From Section~\ref{sec:prob} we obtain the following result which
covers the main statements of \cite[Theorem 3]{jourdain2001american}.
\begin{corollary}\label{cor:bs}
  \begin{enumerate}[label=(\roman*)]
  \item It holds that
    \[C_*\subseteq \graph(t_*)^c \quad\text{and for all
        $(t,x)\in [0,T]\times \R_+$:}\quad u(t,x) \leq v(t,x).\]
  \item If there exists a continuous function
    $\theta\colon \R_+\to [0,T]$ such that for all $x\in \R_+$,
    $\theta(x) \in t_*(x)$, then for all $(t,x)\in [0,T]\times \R_+$
    with $t\leq \theta(x)$,
    \[ u(t,x) = v(t,x) = \EE_{t,x}[ e^{- r(\tau-t)} g(X_\tau)],\]
    where
    $\tau:= \inf\{s\geq t\,\colon s \geq \theta(s)\}\in \cT_{t,T}$ is
    the optimal stopping time.
  \end{enumerate}
\end{corollary}
\begin{proof}
  The first part is an application of Proposition~\ref{prop:ubounds}
  and Corollary~\ref{cor:contreg}, the second part of
  Corollary~\ref{cor:ctspath}.
\end{proof}

In the Black-Scholes model also a more refined analysis has been
carried out on the inverse problem: Given an American payoff function
$f$, does there exist a European option with payoff $g$ and value
function $v$ such that for all $x\in \R_+$
\[f(x) = \inf_{0\leq t\leq T} v(t,x). \] For a class of choices of
$f$, this question could be answered positively and $g$ has been
explicitly. E.\,g. in \cite{jourdain2002approximation} this has been
used as an approximation for prices of American put options in the
Black-Scholes prices. A detailed characterization in terms of an
existence and verification theorem has been developed in
\cite{lenga2017representable} and \cite{lenkalsor:20} using convex optimization
techniques. In particular it has been shown that the value function of the American Put
(almost) admits a European representation. There are still many exciting open questions in this field.

Let us close this section by briefly discussing the PDE perspective:
We know that $v$ is the unique classical solution of the Kolmogorov
backward PDE
\begin{equation}
  \label{eq:24}
  \ddt v(t,x) + \frac{\sigma^2}{2} x^2 \ddxx v(t,x) + rx \ddx v(t,x) - r v(t,x) = 0,
\end{equation}
with terminal condition $v(T,x) = g(x)$, and that $u$ is the solution
of the free boundary problem, cf. e.\,g.\cite{peskir2006optimal},
\begin{equation}
  \label{eq:25}
  \begin{alignedat}{2}
    - \ddt u(t,x) - (\tfrac{\sigma^2}{2} x^2 \ddxx u(t,x) + rx \ddx u(t,x) - r u(t,x)) &\geq 0,&\qquad (t,x)&\in [0,T]\times \R_+\\
    - \ddt u(t,x) - (\tfrac{\sigma^2}{2} x^2 \ddxx u(t,x) + rx \ddx
    u(t,x) - r u(t,x)) &= 0,&\quad (t,x)&\in C \\
    u(t,x) &= f(x),&(t,x)&\in D\\
    C= \{(t,x)\in [0,T]\times E\colon u(t,x)>f(x)\},\quad D&=C^c &&
  \end{alignedat}
\end{equation}
with the smooth fit condition $\ddx u(t,x) = \ddx f(x)$ for
$x\in \partial C$. Therefore, if $x\mapsto \min t_*(x)$ is continuous,
then Theorem~\ref{thm:PDEequality} offers a purely analytic viewpoint
and proof for the respective statements in \cite[Theorem
3]{jourdain2001american} and Corollary~\ref{cor:bs}.

\section{Non-linear Levy Model and uncertain volatility}
\label{sec:nonlinLevy}
The natural extension of the Black-Scholes model to a non-linear
framework is to allow for uncertainty of the volatility
parameter. This can be studied in the framework of non-linear Brownian
motions, see for instance \cite{peng2004nonlinear}. In this section, we will
go a step further and demonstrate that the results from
Section~\ref{sec:prob} are applicable to non-linear Levy processes. We
will mainly rely on~\cite{neufeld2017nonlinear} in which the dynamic
programming principle and the HJB equations for a class of non-linear
Levy processes have been
discussed.

\subsection{Setup}

Let $\Omega = D(\R_+; \R)$ be equipped with the Skorokhod topology,
$\F = \borel{\Omega}$ and let $X$ be the coordinate process,
i.e. $X_t(\omega) = \omega_t$ for each
$(t,\omega)\in [0,\infty)\times \Omega$. By $(\F_t)$ we denote the raw
filtration generated by $X$. We define $\cT_{s,t}$ as the set of all
stopping times with values in $[s,t]$ and set $B_t := e^{r t}$,
$t\geq 0$, for an interest rate $r\in \R$.

Let $\cL$ is the set of all L\'evy measures on $\R$. Set
\begin{align*}
  \fP:= \{\PP\in \cP(\Omega,\FF) &\,\colon\, X \text{ is a
                                   semimartingale on
                                   $(\Omega, \F, (\F_t),\PP)$}\\
                                 &\text{ with characteristics $(B,C,\nu)\ll \d t$ $\PP$-a.s.}\}
\end{align*}
In the following, we will associate the absolutely continuous
semimartingale characteristics $(b_t\d t, c_t \d t, F_t \d t)$ with
the L\'evy triplets $(b,c,F)$ which are taking values in
$\R \times \R_+ \times \cL$. Here, and in the following we keep a
truncation function $h$ fixed, that is a bounded measurable function
$h\colon \R\to\R$ such that $h(z) = z$ on $[-1,1]$.  We will allow for
parameter uncertainty within a set
\[\Theta \subset \R \times \R_+ \times \cL \]
which as assumed to be Borel measurable, non-empty and such that
\begin{equation}
  \label{eq:5}
  \sup_{(b,c,F)\in \Theta}\left[ \int_\R \abs{z}\wedge \abs{z}^2 F(\d
    z) + \abs{b} + \abs{c}\right] <\infty,
\end{equation}
and
\begin{equation}
  \label{eq:4_sup}
  \lim_{\epsilon\to 0} \sup_{(b,c,F)\in \Theta} \int_{\abs{z}\leq
    \epsilon} \abs{z}^2 F(\d z) <\infty.
\end{equation}

\subsection{Nonlinear expectations and Levy processes}
We now keep $T>0$ fixed. For $(t,x)\in [0,T]\times \R$, we define
\begin{equation}
  \label{eq:1_sec5}
  \cP(t,x) := \{ \PP\in \fP \vert X_s = x \text{ $\PP$-a.s. for $s\leq t$ and }
  (b^\PP, c^\PP, F^\PP) \in \Theta\, \PP\otimes \d t \text{-a.\,e. on $[t,\infty)$}\}.
\end{equation}
For a random variable $Y \colon \Omega \to \R$ and
$(t,x) \in [0,T] \times \R$ we set
\begin{equation}
  \label{eq:6_sup}
  \mathcal{E}_{t,x} := \sup_{\PP\in \cP(t,x)} \EE^{\PP}[Y].
\end{equation}

It is immediate to see that $(\cE_{t,x})_{(t,x)\in [0,T]\times E}$
satisfies \ref{i:monoton} and \ref{i:const}. Even more, $\cE_{t,x}$ is
sublinear. The proofs of the following results rely on
\cite{neufeld2017nonlinear} and will presented in
Subsection~\ref{ssec:nL_proofs} below.

\begin{lemma}\label{lem:txLevy}
  For $(t,x)\in [0,T]\times E$, the canonical process is a
  $\cE_{t,x}$-L\'evy process started at $(t,x)$ in the sense that the
  following holds true:
  \begin{enumerate}[label=(\alph*)]
  \item $\{X_t = x\}$ is a full $(t,x)$-scenario
  \item For all $s,\,r\geq 0$ and $\phi \in \cL^\infty(\R)$:
    \[\cE_{t,x}[\phi(X_{t+s+r} - X_{t+s})] =
      \cE_{t,x}[\phi(X_{t+r})]\]
  \item For all $s\geq t$, $r\geq 0$, $n\in \N$ and
    $t\leq s_1\leq ...\leq s_n\leq s$, and
    $\phi \in \cL^\infty(\R^{n+1})$:
    \begin{align*}
      & \cE_{t,x}[\phi(X_{s+r} - X_s, X_{s_1}, \ldots, X_{s_n})]\\
      &\qquad\qquad= \cE_{t,x}[\cE_{t,x}[\phi(X_{s+r} - X_s, x_1,..., x_n)]\vert_{x_1=X_{s_1},\ldots, x_n=X_{s_n}}].
    \end{align*}
  \end{enumerate}
\end{lemma}
\begin{rmk}
  The definition is generalizing the definition of non-linear L\'evy
  processes introduced in \cite{peng2004nonlinear}. If $\Theta$ is a
  singleton then the $X$ is a Levy process in the classical sense.
\end{rmk}

\begin{proposition}
  \label{prop:nL_mb_dpp}
  For each $g\colon \R \to \R$ bounded and Lipschitz continuous, the
  map $(t,x) \mapsto \cE_{t,x}[g(X_T)]$ is measurable and fulfills the
  dynamical programming principle. More precisely, $(\cE_{t,x})$
  satisfies~\ref{i:meas} and~\ref{i:dpp}.
\end{proposition}

We keep fixed $g\colon \R \to \R$ bounded and Lipschitz continuous and
consider the European option prices in the nonlinear Levy model
\[v(t,x) := e^{rt}\cE_{t,x}[e^{-rT} g(X_T)].\]

\begin{proposition}
  \label{prop:pdeLevy}
  $v$ is bounded and continuous and the unique viscosity solution of
  the equation
  \begin{multline}
    \label{eq:pdeLevy}
    -\ddt v(t,x) + rv(t,x)\\
    - \sup_{(b,c,F)\in \Theta} \left[\tfrac{c}{2} \ddxx v(t,x) + b
      \ddx v(t,x) + \int_\R v(t,z) - v(t,0) - \ddx v(t,0) h(z) F(\d z)
    \right] = 0,
  \end{multline}
  with terminal condition $v(T,x) = g(x)$, $x\in E$.
\end{proposition}

\subsection{Optimal stopping and games}

Define
\[
  h(x):= \sup_{0\leq t\leq T} v(t,x)\] and
\[
  w(t,x) := \inf_{\tau \in \cT_{t,T}} e^{rt}
  \cE_{t,x}[e^{-r\tau}f(X_\tau)].\] Moreover, set for $x\in E$,
\begin{alignat}{2}
  t^*(x) &:= \argmax \{v(t,x)\,\colon \, t\in [0,T]\},&\qquad \theta^*(x) &:= \max t^*(x),\\
  C^*&:= \{(t,x)\in [0,T]\times E\,\vert\, t<\theta^*(x)\},&\qquad D^*
  &:= (C^*)^c.
\end{alignat}

$w$ is the value function a stochastic game with reward process
$h(X_t)$. The first player selects the control and thus the dynamics
of $X$ by choosing $\PP\in \cP(t,x)$. The second player then decides
when to stop the game at time $\tau$.

A typical setup for such games are American options with volatility
uncertainty: For instance, assume that the seller selects the exercise
time $\tau$ and
\[ \Theta := \{(r-\frac12 \sigma^2, \sigma^2, 0)\colon \sigma \in
  [\usigma,\bsigma]\}\] for $\bsigma$, $\usigma\in (0,\infty)$ with
$\usigma<\bsigma$. Then, $w$ becomes the issuers value function of the
issuers value function of the American option under volatility within
the range $[\usigma,\bsigma]$, cf. \cite[Section 5]{nutz2015optimal}.

\begin{lemma}
  \label{lem:wconststopping} Let $(s,x)\in [0,T]\times E$ be such that
  $w(s,x) = h(x)$. Then, $w(t,x) = h(x)$ for all $t\in [s,T]$.
\end{lemma}
The proof is shown below in Section~\ref{ssec:nL_proofs} below.

\begin{proposition}
  \label{prop:EAO_Levy}
  \begin{enumerate}[label=(\alph*)]
  \item\label{i:EAO_Levy_onesided} It holds that for all
    $(t,x) \in [0,T] \times E$,
    \[ v(t,x) \leq w(t,x) \quad\text{and}\quad C^* \subseteq (\graph
      t^*)^c.\]
  \item\label{i:EAO_Levy_twosided} If, in addition,
    \[\Theta \subset \{(b,c,F) \in \R\times \R_+\times \cL \vert F
      \equiv 0\}\] and $\theta$ is continuous, then for all
    $(t,x) \in [0,T]\times E$,

    \[ v(t,x) = w(t,x).\]

  \end{enumerate}
\end{proposition}
\begin{proof}
  By Proposition~\ref{prop:nL_mb_dpp} we can apply
  Proposition~\ref{prop:ubounds} and Corollary~\ref{cor:contreg} to
  obtain Proposition~\ref{prop:EAO_Levy}.\ref{i:EAO_Levy_onesided}.
  Item~\ref{i:EAO_Levy_twosided} then follows from
  Corollary~\ref{cor:ctspath}.
\end{proof}

\begin{theorem} \label{prop:fbpBachelier} Assume that
  \[\Theta \subset \{(b,c,F) \in \R\times \R_+\times \cL \vert F
    \equiv 0\},\] and that $\theta$ is continuous. Then, $w$ is a
  viscosity solution of the free boundary problem
  \begin{equation}
    \label{eq:fbpBachelier}
    \begin{alignedat}{3}
      &-\ddt w(t,x) + rw(t,x)\hspace{-22em}&\hspace{22em} &&&\\
      &&- \sup_{(b,c,0)\in \Theta} \left[\tfrac{c}{2} \ddxx
        w(t,x) + b \ddx w(t,x)\right]&\leq 0,&\quad (t,x) &\in [0,T]\times E,\\
      &-\ddt w(t,x) + rw(t,x)\hspace{-22em}&\hspace{22em}&&&\\
      &&- \sup_{(b,c,0)\in \Theta} \left[\tfrac{c}{2} \ddxx
        w(t,x) + b \ddx w(t,x)\right]&= 0,&\quad (t,x) &\in C,\\
      &&w(t,x) &= h(x),& (t,x) &\in C^c,\\
      &&C= \{(t,x) \in [0,T)\times E &\colon w(t,x) < h(x)\},&&\\
      &&\text{(continuous fit)}\qquad w(t,x) &= h(x),& (t,x) &\in
      \partial C,
    \end{alignedat}
  \end{equation}
  and, if $v$ and $\theta$ are $C^1$, then also the smooth fit
  condition holds true
  \[\nabla_x w(t,x) = \nabla_x h(x),\qquad (t,x) \in \partial C.\]
\end{theorem}
\begin{proof}
  By Lemma~\ref{lem:wconststopping}, $w(t,x) = h(x)$ for
  $(t,x)\in D:= C^c$. The statement follows by
  Proposition~\ref{prop:pdeLevy} and Theorem~\ref{thm:PDEequalitySup}.
\end{proof}

\subsection{Proofs}
\label{ssec:nL_proofs}
To obtain the desired properties for $(t,x) \mapsto \cE_{t,x}$ we want
to apply the setup from \cite{neufeld2017nonlinear}. Define
$\Omega_0:= \{ \omega\in \Omega\,\colon\, \omega(0) = 0\}$ and let
$X^0$ be the restriction of $X$ to $\Omega_0$, and denote by
$(\F^0_t)$ we denote the raw filtration generated by $X^0$. We now
define for $(t,x)\in [0,\infty) \times \R$ the mapping
$\theta(t,x) \colon \Omega_0 \to \Omega$,
\[ \theta(t,x)(\omega) := x + \omega((.-t) \wedge 0).\] $\theta(t,x)$
is measurable.

Moreover, it has the left inverse
\[ \zeta(t,x)(\omega):= \omega(.+t) - \omega(t).\] Indeed, for
$\omega\in \Omega_0$,
\[ \zeta(t,x)(\theta(t,x)(\omega)) = \zeta(t,x)(x + \omega((.-t)\wedge
  0)) = x + \omega - x = \omega.\] We define by $\cP_0$ the set of
restricts of $\PP$ to $(\Omega_0, \F^0)$, for all $\PP\in \cP(0,0)$
and set
\[\cE^0[Y] := \sup_{\PP \in \cP_0} \EE^\PP[Y]\]
for $Y\in \cL^\infty(\Omega_0)$.

\begin{lemma}\label{lem:linkToNN}
  For each $Y\in \cL^\infty(\Omega)$ and $(t,x)\in [0,T] \times E$, it
  holds that
  \begin{equation}
    \label{eq:linkToNN}
    \cE_{t,x}[Y] = \cE^0[Y\circ \theta(t,x)].
  \end{equation}
  In particular, for each $N\in \N$ and
  $t\leq s_1 \leq \ldots\leq s_n\leq T$ and each $\phi\in C_b(\R^n)$:
  \[ \cE_{t,x}[\phi(X_{s_1},\ldots, X_{s_n})] =
    \cE^0[\phi(x+X^0_{s_1-t},\ldots, x+X^0_{s_n-t})]\]
\end{lemma}
\begin{proof}
  The duality
  \[ \sup_{\PP \in \cP(t,x)} \EE^\PP[Y] = \sup_{\PP \in \cP_0}
    \EE^\PP[Y\circ \theta(t,x)]\] is established by the push-forward
  measures $\theta(t,x)\sharp \PP^0$ for $\PP^0\in \cP_0$ and
  $\zeta(t)\sharp \PP$ for $\PP\in \cP(t,x)$.
\end{proof}

This characterization allows us to apply the results from
\cite{neufeld2017nonlinear}.

\begin{proof}[Proof of Lemma~\ref{lem:txLevy}]
  By \cite[Theorem 2.1]{neufeld2017nonlinear}, $(X_t^0)_{t\geq 0}$ is
  a non-linear Levy process under $\cE^0$. Using this and
  Lemma~\ref{lem:linkToNN} we get for all $n\in \N$,
  $\phi \in \cL^\infty(\R^{n+1})$ and all $s\geq t$ and
  $t\leq s_1 \leq \ldots \leq s_n\leq s$:
  \begin{align*}
    \vspace*{2em}&\vspace*{-2em}
                   \cE_{t,x}[\phi(X_{s+r} - X_s, X_{s_1},\ldots, X_{s_n})]\\
                 &= \cE^0[\phi(X^0_{s+r-t} -
                   X_{s-t}^0, x+ X^0_{s_1-t},\ldots,x+ X^0_{s_1-t})] \\
                 &= \cE^0[ \cE^0[\phi(X^0_{s+r-t} -
                   X_{s-t}^0, y_1,\ldots, y_n)]\vert_{y_1= x+X_{s_1-t}^0,\ldots,y_n= x+X_{s_n-t}^0}] \\
                 &=\cE^0[\cE_{t,x}[\phi(X_{s+r} - X_s,y_1,\ldots,y_n)]\vert_{y_1=
                   x+X_{s_1-t}^0,\ldots,y_n= x+X_{s_n-t}^0}]\\
                 &= \cE_{t,x}[\cE_{t,x}[\phi(X_{s+r} -
                   X_s,y_1,\ldots,y_n)]\vert_{y_1=X_{s_1},\ldots,y_n = X_{s_n}}].
  \end{align*}
  This yields the independence of increments. Similarly, for $\phi \in \cL^\infty(\R)$ and all $s\geq t$ and
  $r\geq 0$,
  \begin{multline}
    \cE_{t,x}[ \phi(X_{s+r} - X_s)] = \cE^0[\phi(x+X_{s+r-t}^0 -
    X_{s-t}^0)] = \cE^0[\phi(x+X_{r}^0)]\\
    = \cE^0[\phi(x+X_{r}^0)]= \cE^0_{t,x}[\phi(X_{t+r})]
  \end{multline}
  Hence, $(X_t)$ is a non-linear Levy process in the sense of
  Lemma~\ref{lem:txLevy}.
\end{proof}

For $\tau\in \cT_{0,T}$ and $\omega$, $\tilde \omega \in \Omega_0$,
define
\begin{equation}
  \label{eq:10_tilde}
  (\omega \otimes_\tau \tilde\omega)(t) := \omega(t) \1_{\tau(\omega)<t} + (\tilde \omega(t-\tau(\omega)) + \omega(\tau(\omega))) \1_{\tau(\omega)\geq t}.
\end{equation}

In particular, we have that
\[ \1_{[t,\infty)} \otimes \tilde \omega = \theta(t,0)(\tilde
  \omega)\] for all $\tilde \omega\in \Omega^0$.

The dynamic programming principle will be derived from the following
result on Levy processes on $(\Omega^0, \F^0,\mathcal{\cE^0})$.
\begin{lemma}[Theorem 2.1.(ii) in \cite{neufeld2017nonlinear}]
  \label{lem:dppnonlinlevy}
  Let $Y\in \cL^\infty(\Omega^0)$ and $\sigma$, $\tau\in \cT_{0,T}$ be
  such that $\sigma\leq \tau$. Then, for all $\omega\in \Omega^0$:
  \[ \cE^0[Y(\omega \otimes_\sigma .)] = \cE^0 [ \cE^0[Y(\tilde \omega
    \otimes_\tau .)]\vert_{\tilde \omega = (\omega\otimes_\sigma
      .)}.\]
\end{lemma}

We are now able to show the dynamic programming principle in the
formulation of~\ref{i:dpp}.

\begin{proof}[Proof of Proposition~\ref{prop:nL_mb_dpp}]
  From Lemma~\ref{lem:linkToNN} we get for $g\colon E\to \R$ bounded
  and Lipschitz continuous, that
  \begin{equation}
    \label{eq:2_bs}
    \cE_{t,x}[ g(X_T)] = \cE^0[g(X_T\circ \theta(t,x))] = \cE^0[g(x + X_{T-t}^0)].
  \end{equation}  
  By \cite[Theorem 2.5]{neufeld2017nonlinear}, this is continuous as a
  function of $(t,x)$ and thus measurable. Let us fix
  $\tau\in \cT_{t,T}$. Note that for $s\geq t$,
  \begin{equation*}
    \{\omega\in \Omega_0 \,\colon\, \tau(\theta(t,x)(\omega)) \leq s\} = (\theta(t,x)^{-1}\circ \tau^{-1})([0,s])
  \end{equation*}
  Since $\tau$ is an $(\F_t)_{t\in[0,T]}$ stopping time, it holds that
  $\tau^{-1}([0,s])\in \F_s$. Now, $\theta(t,x)$ is
  $\F_{(s-t)\vee 0}^0\slash \F_s$ measurable, and hence
  \[\tau_{t,x}:= \tau\circ \theta(t,x)-t\]
  is an $(\F_t^0)_{t\in [0,T]}$ stopping time with values in
  $[0, T-t]$.
  Now, define for $\omega\in \Omega_0$:
  \[ Y(\omega) := B_{T}^{-1}g(X_T(\theta(t,x)(\omega))) = B^{-1}_{T}
    g(x+\omega(T-t)).\] Since $\tau_{t,x}$ takes values in $[0,T-t]$,
  \begin{align*}
    Y(\omega\otimes_{\tau_{t,x}} \tilde \omega)
    &= B^{-1}_{T} g(x+ (\omega\otimes_{\tau_{t,x}}\tilde \omega)(T-t))\\
    &=  B^{-1}_{T}g(x+\omega(\tau_{t,x}(\omega)) +
      \tilde\omega(T-(\tau_{t,x}(\omega)+t)))\\
    &=  B^{-1}_{T}g(x+X^0_{\tau_{t,x}(\omega)}(\omega) +
      X^0_{T-(\tau_{t,x}(\omega)+t)}(\tilde \omega)).
  \end{align*}
  Finally, we use that $\tau_{t,x}$ is a stopping time as well as
  Lemma~\ref{lem:linkToNN} and~\ref{lem:dppnonlinlevy} to get that
  \begin{align*}
    \cE_{t,x}[B_{T}^{-1}g(X_T)]
    &= \cE^0[B_{T}^{-1}g(x+X_{T-t}^0)] \\
    &= \cE^0\left[\cE^0[
      B^{-1}_{T} g(y+X^0_{T-s})]\vert_{s=T-(\tau_{t,x}+t), y =
      x+X_{\tau_{t,x}}^0}\right]\\
    &= \cE^0\left[ B_{\tau_{t,x}}^{-1}B_{t+\tau_{t,x}}\cE^0[B_T^{-1}
      g(y+X_{T-s})]\vert_{s=T-(\tau_{t,x}+t), y =
      x+X_{\tau_{t,x}}^0}\right] \\
    &= \cE^0[ B^{-1}_{\tau_{t,x}}v(\tau_{t,x}+t,
      x+X_{\tau_{t,x}}^0)] = \cE_{t,x}[B_{\tau}^{-1} v(\tau, X_\tau)].  \qedhere
  \end{align*}
\end{proof}

\begin{proof}[Proof of Proposition~\ref{prop:pdeLevy}]
  In the case $r=0$, \cite[Theorem 2.5]{neufeld2017nonlinear} states
  that $v$ is the unique (continuous) viscosity solution
  of~\eqref{eq:pdeLevy}. The extension to $r\in \R$ is straight
  forward. This yields Proposition~\ref{prop:pdeLevy}.
\end{proof}

\begin{proof}[Proof of Lemma~\ref{lem:wconststopping}]
  We show the statement by contradiction. Assume there exists a
  $t\in [s,T]$ such that
  \[ w(t,x) < h(x).\] Let $\tau \in \cT_{t,T}$ be such that
  \[ B_t \cE_{t,x}[B_\tau^{-1} h(X_\tau)] < h(x).\] Define for
  $\omega\in \Omega$
  \[ \varsigma(\omega) := \tau(\omega((.-(t-s))\vee 0)) + t-s\] Then,
  $\varsigma$ is a stopping time with values in $\cT_{s,T}$, and
  \[ (\varsigma \circ \theta_{s,x})(\omega) = \tau\circ(t,x) + t-s\]
  Using this, the positive homogeneity of $\cE^0$ and that
  $B_t = e^{rt}$, we get
  \begin{align*}
    w(s,x)&\leq B_s\cE_{s,x}[B_{\varsigma}^{-1}h(X_\tau)]\\
          &= B_s
            \cE^0[B_{\varsigma\circ \theta(s,x)}^{-1} h(x+X_{\varsigma\circ
            \theta(s,x) + s})]\\
          &= B_s \cE^0[B_{\tau\circ \theta(t,x)+t-s}^{-1} h(x+X_{\tau\circ
            \theta(t,x) + t})]\\
          &= B_s B_{t-s}^{-1}\cE^0[B_{\tau\circ \theta(t,x)}^{-1} h(x+X_{\tau\circ
            \theta(t,x) + t})]\\
          &= B_t\cE_{t,x}[B_{\tau}^{-1} h(x+X_{\tau})]\\
          &<h(x).
  \end{align*}
  This contradicts the assumptions on $(s,x)$. Hence, such a
  $t\in [s,T]$ cannot exist.
\end{proof}

\section{Finite State Space}
\label{sec:finite}
We consider now the case where $E$ is finite, more precisely
$d:= \abs{E} \in \N $. Without loss of generality we assume that
$E = \{1,...,d\}$.  Let us fix $T\in (0,\infty)$ and let
$\Omega := D([0,T];E)$ equipped with the Borel $\sigma$-algebra
$\F$. $X$ is the coordinate process generating the raw filtration
$(\F_t)_{t\in [0,T]}$. We denote by $\cT_{s,t}$ the set of
$[s,t]$-valued stopping times which take values in $[s,t]$.

We fix a compact set $\cS$ of stochastic matrices in $\R^{d\times
  d}$. For a $Q\in \cS$ fixed let
$(\PP_{t,x})_{(t,x)\in [0,T]\times E}$ be the strong Markov family
which satisfies that $X_s = x$ $\PP_{t,x}$-almost surely for all
$s\in [0,t]$, and that
\[ \left(g(X(s)) - g(x) - \int_t^s Q g(X(r)) \d r\right)_{s\in
    [t,T]}\] is a $\PP_{t,x}$-martingale for every $g\in C(E)$. We
define for fixed $g\in C(X)$ and $c \in C([0,T], E)$
\[v(t,x):= v(t)_x := \EE^{\PP_{t,x}}[g(X_T) + \int_t^T c(s,X(s)) \d
  s].\] The value function $v$ is the unique solution of the linear
ODE
\[ \ddt v(t) + Qv(t) - c(t,x) = 0,\quad t\in [0,T],\qquad v(T) = g.\]
In particular, $v$ is continuous and Proposition~\ref{prop:ubounds},
Corollary~\ref{cor:contreg} and Theorem~\ref{thm:wvu} are
applicable.

\begin{rmk}
  In this framework, it is easily shown in that if $d=2$, then
  the map from $g$ to $f$ is surjective, also the case $d=3$ can be partially treated. For arbitrary $d$
  surjectivity was shown up to time $T< \frac1{\norm{Q}{}}$. More precisely, for each $Q\in \cS$,
  $T\in (0, 1/\norm{Q}{})$ and each $f\in \R^d$, there exists a
  $g\in \R^d$ such that
  \[ \inf_{0\leq t\leq T} e^{tQ}g = f, \] where the infimum is taken
  element-wise. The same holds true when replacing $\inf$ by $\sup$. These results will be proved elsewhere.
\end{rmk}

Next, we want to have a look into a non-linear framework for
continuous-time Markov chains, following~\cite{nendel2018markov}.

Denoting by $\Pi_{s,t}$ the set of all partitions of $[s,t]$, we
define for $(t,x)\in [0,T]\times E$, $\fP(t,x)$ to be the set of all
$\PP\in \cP(\Omega)$ such that
\begin{enumerate}[nolistsep]
\item it holds that $\PP[X_s = x] = 1$, for all $s\in [0,t]$,
\item there exist $\pi \in \Pi_{t,T}$ with
  $\pi =: \{t=t_0,t_1,...,t_n=T\}$ for some $n\in \N$,
  $\mathscr{Q}_i \colon E \to \cS$, $i=1,...,n$,
  \[ \left(g(X(s)) - g(x) - \int_t^{s} Q^\PP(r)g(X(r)) \d
      r\right)_{s\in [t,T]}\] is a $\PP$-martingale for each
  $g\in C(E)$, where for $r\in [t,T]$,
  \[ Q^\PP (r) := \sum_{i=1}^n\1_{(t_{i-1}, t_i]}(r)
    \mathscr{Q}_i(X(t_{i-1})). \]
\end{enumerate}

As space of integrable variables we take $\cH:=
\cL^\infty(\Omega)$. Now, we define the non-linear expectations for
$Y\in \cH$ as
\begin{equation}
  \label{eq:3_expectation}
  \cE_{t,x}[Y] := \sup_{\PP \in \fP(t,x)} \EE^{\PP}[Y],\qquad (t,x)\in [0,T]\times E.
\end{equation}

We fix a terminal payoff $g \in C(E)$, set the running costs
$\varrho \equiv 0$, keep $B\equiv 1$ and define the value function
\begin{equation}
  \label{eq:4_expectation}
  v(t,x) := v(t)_x := \cE_{t,x}[g(X_T)],\qquad (t,x)\in [0,T]\times E.
\end{equation}

Note that, by construction, $v$ can be expressed in terms of the Nisio
semigroups considered in~\cite[Definition 3.2]{nendel2018markov}. More
precisely,
\begin{equation}
  \label{eq:10_sup}
  \begin{aligned}
    v(t) &= \sup_{\{t_0,t_1,...,t_n\}:=\pi\in \Pi_{t,T}} \sup_{Q^1\in \cS} e^{(t_1-t_0)Q^1}\left(...\left( \sup_{Q^n\in \cS} e^{(t_n-t_{n-1})Q^{n}}g\right)\right) \\
    &= \lim_{n\to\infty} \sup_{Q^{1}\in \cS}
    e^{(T-t)2^{-n}Q^1}\left(...\left( \sup_{Q^{2^n}\in \cS}
        e^{(T-t)2^{-n}Q^{2^n}}g\right)\right),
  \end{aligned}
\end{equation}
where the $\sup$ are taken point-wise and the second equality follows
from \cite[Corollary 3.8]{nendel2018markov}.

\begin{proposition}
  The family $(\cE_{t,x})_{(t,x)\in [0,T]\times E}$ fulfills
  \ref{i:const}, \ref{i:monoton}, \ref{i:monoton_strict},
  \ref{i:txeval}, and \ref{i:meas}.
\end{proposition}
\begin{proof}
  \ref{i:const}, \ref{i:monoton}, \ref{i:monoton_strict},
  \ref{i:txeval} are clear by definition and properties of $\sup$ and
  (linear) expectations. Also, $(t,x)\mapsto v(t,x)$ is continuous
  since $E$ is finite,
  \[(t,Q)\mapsto (e^{(T-t)Q}g)_x\] is continuous for each $x\in E$ and
  $\cS$ is separable. This yields
  \ref{i:meas}.
  \end{proof}

By \cite[Theorem 1.3]{nendel2018markov}, the HJB equation for $v$
becomes the backward ODE
\begin{equation}
  \label{eq:HJBODE}
  \ddt v(t) + \sup_{Q\in \cS} Qv(t) = 0,\qquad v(T) = g,
\end{equation}
for $t\in [0,T]$. However, $Q$ is in general not local and
Section~\ref{sec:PDE} are not applicable here, even not for the linear
case where $\cS$ is a singleton.

Since $\cS$ is bounded, Picard-Lindel\"of theorem yields that
\eqref{eq:HJBODE} has a unique solution which is
$C^1$. In~\cite{nendel2018markov} also the case with $\varrho \leq 0$ is covered.

We define the function $h$ by
\[ h(x):= \sup_{0\leq t\leq T} v(t,x),\qquad \text{for }x\in E\] and
focus now on the value function $w$ of the minimization problem
\begin{equation}
  \label{eq:11}
  w(t,x) := \inf_{\tau\in \cT_{t,x}} \cE[h(X_\tau)],\qquad (t,x)\in
  [0,T]\times E.
\end{equation}
Note that $h$ is continuous and thus measurable, in
particular~\ref{i:meas} is satisfied.

Recall that the stopping region is defined $D$ as
\[D:= \{ (t,x) \in [0,T]\times E\, \colon w(t,x) = h(x)\}.\] We
trivially get a one-sided version of~\ref{i:dpp}: for each
$\tau \in \cT_{t,T}$:
\begin{multline}
  \cE_{t,x}[v(\tau, X_\tau)] = \sup_{\PP\in \fP(t,x)} \EE^\PP[
  \sup_{\QQ \in \fP(\tau,X_\tau)}\EE^\QQ[g(X_T)]]\\ \geq \sup_{\PP\in
    \fP(t,x)} \EE^\PP[ \EE^\PP[g(X_T)]] = \sup_{\PP\in \fP(t,x)}
  \EE^\PP[g(X_T)]= v(t,x).
\end{multline}

Thus, Proposition~\ref{prop:ubounds} and Corollary~\ref{cor:contreg},
see also Remark~\ref{rmk:onesideddpp}, yield that $v\leq w$ and
$\graph(t^*) \subseteq D$, where
\[t^*(x) := \argmax\{v(t,x)\,\colon t\in [0,T]\}.\]

If we additionally assume that the other part of the dynamic programming principle
holds true, namely that for each $\tau \in \cT_{t,T}$:
\[ v(t,x) =\cE_{t,x}[v(\tau, X_\tau)],\] then also
Theorem~\ref{thm:wvu} is applicable.

\begin{appendix}
  \section{Notation}
  \label{sec:notation}
  For $t\leq T$ denote by $\cT_{t,T}$ the set of all stopping times
  $\tau$ such that almost surely $t\leq \tau \leq T$. For a measurable
  space $(\Omega,\F)$ we denote by $\cL^0(\Omega)$ the vector space of
  all measurable functions, and by $\cL^\infty(\Omega)$ the space of
  all real valued bounded and measurable functions on $\Omega$. For
  $N\in \N$ let $S_N$ be the space of symmetric $N\times N$ matrices.

  For a metric space $(E,d)$ and $\epsilon \in (0,\infty)$, $x\in E$
  we denote the open ball of radius $\epsilon$ around $x$ by
  $B_E(x,\epsilon)$.

  For sets $A$ and $B$ function $F\colon A\to B$ we define the graph
  of $F$ as
  \[ \graph(F) := \{(a,b) \in A\times B \colon F(a) = b\}. \]

  For a measurable space $(O,\F)$ and a metric space $E$ we denote by
  $L^p(O; E)$, $p\in [0,\infty]$ the Lebesgue spaces.

  For Banach spaces $E$, $F$, $G$ we denote by $L(E;G)$ the space of
  linear continuous maps from $E$ to $G$ and by $L_{(2)}(E,F;G)$ the
  space of bilinear continuous maps from $E\times F$ into $G$.

  \section{Viscosity solutions of PDEs}
  \label{sec:visco}
  In this section we briefly recall the definition we fix $d\in \N$
  and $T>0$. We consider a \emph{proper} function
  $F\colon \times \R^d\times \R \times \R^d \times S_{d} \to \R$,
  which is means that for all $r, s\in \R$, $x,p\in \R^d$,
  $X, Y\in S_d$ with $r\leq s$ and $Y\leq X$ it holds that
  \begin{equation}
    \label{eq:proper}
    F(x,r,p,X) \leq  F(x,s,p,Y).
  \end{equation}

  For such a function, we consider the non-linear PDE
  \begin{equation}
    \label{eq:pdeproper}
    -\ddt u(t,x) + F(x, u(t,x), D u(t,x), D^2u(t,x)) = 0.
  \end{equation}

\begin{definition}\label{def:visc}
  Let $F$ be a proper function, $\cO\subseteq [0,T]\times \R^d$ and
  $T\in (0,\infty)$.
  \begin{enumerate}
  \item A \emph{viscosity subsolution} of~\eqref{eq:pdeproper} on
    $\cO$ is a an upper semicontinuous function $u\colon \cO \to \R$
    such that for all $(t,x)\in \cO$, $\phi \in C^{1,2}(\cO;\R)$ such
    that $u(t,x) = \phi(t,x)$ and $u\geq \phi$ in an
    $\cO$-neighborhood of $(t,x)$, it holds that
    \begin{equation}
      \label{eq:subsol}
      -\ddt \phi(t,x) + F(t, x, \phi(t,x), D\phi(t,x), D^2\phi(t,x)) \leq 0.
    \end{equation}
  \item A \emph{viscosity supersolution}~\eqref{eq:pdeproper} on $\cO$
    is a an upper semicontinuous function $u\colon \cO \to \R$ such
    that for all $(t,x)\in \cO$, $\phi \in C^{1,2}(\cO;\R)$ with
    $u(t,x)= \phi(t,x)$ and $u\leq \phi$ in an $\cO$-neighborhood of
    $(t,x)$, it holds that
    \begin{equation}
      \label{eq:subsol_equation}
      -\ddt \phi(t,x) + F(t, x, \phi(t,x), D\phi(t,x), D^2\phi(t,x)) \geq 0.
    \end{equation}
  \item A \emph{viscosity solution} of~\eqref{eq:pdeproper} on $\cO$
    is a continuous function $u\colon \cO\to\R$ which is a sub- and a
    supersolution in the sense of this definition.
  \end{enumerate}
\end{definition}

\begin{rmk}
  The term $-\ddt u$ comes from the fact that we are mainly interested
  in terminal value problems here. We stress the equivalence to of the
  solution concepts when switching to the initial value problem by
  time change: if $u$ is a sub-/supersolution in the sense above, then
  \[ \tilde u(t,x) := u(T-t,x)\] is a sub-/supersolution if $-\ddt u$
  is replaced by $\ddt \tilde u$.
\end{rmk}

\end{appendix}

\end{document}